\documentclass[10pt]{amsart}
\usepackage{amssymb}
\usepackage{epsfig}
\usepackage{url}
\usepackage[T1]{fontenc}
\usepackage{color}
\usepackage{tikz}
\usetikzlibrary{arrows}
\usepackage{longtable}
\usepackage{listings}
\usepackage{xcolor}

\definecolor{codegreen}{rgb}{0,0.6,0}
\definecolor{codegray}{rgb}{0.5,0.5,0.5}
\definecolor{codepurple}{rgb}{0.58,0,0.82}
\definecolor{backcolour}{rgb}{0.95,0.95,0.92}

\lstdefinestyle{mystyle}{
    backgroundcolor=\color{backcolour},   
    commentstyle=\color{codegreen},
    keywordstyle=\color{magenta},
    numberstyle=\tiny\color{codegray},
    stringstyle=\color{codepurple},
    basicstyle=\ttfamily\footnotesize,
    breakatwhitespace=false,         
    breaklines=true,                 
    captionpos=b,                    
    keepspaces=true,                 
    numbers=left,                    
    numbersep=5pt,                  
    showspaces=false,                
    showstringspaces=false,
    showtabs=false,                  
    tabsize=2
}
\theoremstyle{plain}
\makeatletter
\@namedef{subjclassname@2020}{%
  \textup{2020} Mathematics Subject Classification}
\makeatother
\newtheorem{thm}{Theorem}[section]
\newtheorem{cor}[thm]{Corollary}
\newtheorem{lem}[thm]{Lemma}
\newtheorem{prop}[thm]{Proposition}

\def\cal{\mathcal}
\def\bbb{\mathbb}

\renewcommand{\phi}{\varphi}

\newcommand{\N}{\bbb{N}}

\newcommand{\Mod}[1]{\ (\mathrm{mod}\ #1)}
\newcommand{\Par}{\textrm{par}\hspace{0.05cm} }
\newcommand{\T}{{\bf T}}
\newcommand{\TA}{{\bf T}_{\mathcal{A}}}
\newcommand{\Ta}[1]{{\bf T}_{\mathcal{A}_{#1}}}

\allowdisplaybreaks

\begin{document}
\title[Connections between certain numbers related to derangements and $r$-permutations]{Connections between certain numbers related to derangements and $r$-permutations}
\author{Piotr Miska}
\address{Institute of Mathematics, Faculty of Mathematics and Computer Science, Jagiellonian University in Krak\'ow, {\L}ojasiewicza 6, 30-348 Krak\'ow, Poland \newline
and Department of Mathematics, J. Selye University, P. O. Box 54, 945 01 Kom\'arno, Slovakia
}
\email{piotr.miska@uj.edu.pl}
\email{miskap@ujs.sk}

\author{Błażej Żmija}
\address{Charles University, Faculty of Mathematics and Physics, Department of Mathematical Analysis, Sokolov\-sk\' a 83, 18600 Praha~8, Czech Republic \newline
and Institute of Mathematics of the Polish Academy of Sciences, \'{S}niadeckich 8, 00-656 Warsaw, Poland}
\email{blazej.zmija@gmail.com}

\thanks{The research of the first author was partially supported by the grant of the Polish National Science Centre no. UMO-2019/34/E/ST1/00094. The second author was supported by the Charles University grant PRIMUS/25/SCI/017.}

\begin{abstract}
For non-negative integer parameters $r,u,m,n$ define
\begin{align*} 
    \cal{D}(r,u,m,n) &:= \big\{\ \sigma\in \cal{S}_{r+n}\ \big|\ \sigma(x)=y \textrm{ for exactly } u \textrm{ pairs } (x,y) \textrm{ such that } 1\leq x,y\leq r \textrm{ and}  \\
     & \hspace{2.3cm} \sigma(t)=t \textrm{ for exactly } m \textrm{ elements } r+1\leq t\leq r+n\ \big\}, \\
     \cal{D}_{r,u,m}(n) & := \big\{\ \sigma\in \cal{S}_{r+n}\ \big|\ \forall_{1\leq x<y\leq r} \  x \textrm{ and } y \textrm{ are in disjoint cycles of } \sigma \textrm{ and} \\
     & \hspace{2.3cm}  \sigma(z)=z \textrm{ for exactly } u \textrm{ elements } 1\leq z\leq r, \textrm{ and}  \\
     & \hspace{2.3cm} \sigma(t)=t \textrm{ for exactly } m \textrm{ elements } r+1\leq t\leq r+n\ \big\},
\end{align*}
where $\mathcal{S}_{n}$ denotes the set of all the permutations of $\{1,\ldots ,n\}$.

In this paper we study connections between the sets $\cal{D}(r,u,m,n)$, $\cal{D}_{r,u,m}(n)$, and the sets of (some classes of) $r$-derangements. We rely mostly on counting arguments.

\end{abstract}

\maketitle

\section{Introduction}

Denote by $\N$ and $\N_+$ the sets of all non-negative integers and all positive integers, respectively. For every $n\in\N$ we denote by $\cal{S}_n$ the set of all permutations of the set $\{1,\ldots ,n\}$. For each $n\in\N$, let $\mathcal{D}(n)$ denote the set of derangements of the set $\{1,\ldots ,n\}$, that is, all the permutations in $\cal{S}_{n}$ without fixed points:
\begin{align*}
    \mathcal{D}(n):=\big\{\ \sigma\in\mathcal{S}_{n}\ \big|\ \sigma(i)\neq i \textrm{ for all } i\in\{1,\ldots ,n\}\ \big\}.
\end{align*}
Let $d(n):=\#\mathcal{D}(n)$. The sequence $(d(n))_{n\in\N}$, called the sequence of derangements, appears in the context of various number sequences, for example in the paper of Sun and Zagier \cite{SZ}, devoted to Bell numbers. Arithmetic properties of numbers of derangements were studied by Miska in \cite{M}. In the same article arithmetic properties of numbers of even and odd derangements were investigated.

The numbers $d(n)$ were generalized by Wang, Miska, and Mez\H{o} in \cite{WMM}. They introduced the notion of the number of $r$-derangements $d_r(n)$. The number $d_r(n)$ is defined as the number of $r$-derangements of a set $\{1,\ldots ,r+n\}$, i.e., derangements $\rho\in\cal{S}_{r+n}$ such that any two numbers from the set $\{1,\ldots ,r\}$ lie in disjoint cycles of $\rho$ written as a product of pairwise disjoint cycles. In particular, $d_0(n)=d(n)$ and $d_1(n)=d(n+1)$. It is worth noting that not just the numbers of derangements were generalized in that manner. For other examples of $r$-generalizations of some combinatorial objects, see \cite{B, NyR, Sh}.

Capparelli, Ferrari,  Munarini, and Zagaglia Salvi in \cite{CFMZS} studied the generalized derangement numbers $d_n^{(r)}$ defined as numbers of fixed-point-free bijections between two sets with $n+r$ elements that have exactly $n$ elements in common. More generally, they were interested in the generalized recontres numbers $d_{n,m}^{(r)}$ that count bijections between two $(n+r)$-element sets with $n$-element intersection which have exactly $m$ fixed points. Their work was motivated by the article of Beggas, Ferrari, and Zagaglia Salvi \cite{BFZS}, who investigated the numbers of widened permutations, i.e., the numbers of bijections between two sets of cardinality $n+1$ that have exactly $n$-elements in common.

Another generalization of derangements was proposed recently in a paper by Moshtagh and Fallah-Moghaddam \cite{MFM}. They considered the numbers $d(r,n)$ of so-called block derangements, that is, fixed-point-free permutations $\sigma\in\cal{S}_{r+n}$ such that $\sigma(i)\not\in\{1,\ldots ,r\}$ for any $i\in\{1,\ldots ,r\}$. Their main result is that $d(r,n)=r!d_r(n)$ for every $r,n\in\N$.

Motivated by the aforementioned results, we introduce and study two new 
classes of sequences of derangements, where the first one generalizes block derangement numbers $d(r,n)$ and the other one extends the notion of $r$-derangement numbers $d_r(n)$. Moreover, we shall see that generalized recontres numbers $d_{n,m}^{(r)}$ can be represented as sums of certain numbers from the first class (see Proposition \ref{P1}). Our goal is to prove a strict connection between those two classes with the use of a combinatorial argument. We also study some further variations of these sequences by fixing the number of disjoint cycles in the counted permutations, their parity, or both.

We introduce the main definitions in the next section. In Section \ref{SecMainLemma} we define the main tool used in our proofs: the class of cycle-splitting functions $\TA$. We also prove the crucial Lemma \ref{LemMain}. The main result, Theorem \ref{main}, is presented in Section \ref{SecMainResults}. In this part, we rely on counting arguments. A different approach, based on (exponential) generating functions can be found in Section \ref{SecPowerSeries}. In Section \ref{SecDerang} we find strict connections between our new sequences and the sequences of $r$-derangements $(d_{r}(n))_{n\in\mathbb{N}}$. The last section is devoted to study of $r$-derangements with fixed parity. In particular, we present explicit formulae in Theorem \ref{ThmExForDri}.

\section{Notation}

For a permutation $\sigma$ define its parity $\Par \sigma$ as follows:
\begin{align*}
    \Par\sigma :=\left\{\begin{array}{ll}
    0, & \hspace{0.2cm} \textrm{ if } \sigma \textrm{ is even}, \\
    1, & \hspace{0.2cm} \textrm{ if } \sigma \textrm{ is odd}.
    \end{array}\right.
\end{align*}

Let $\cal{D}(r,u,m,n)$ be the set of permutations $\sigma$ of the set $\{1,\ldots,r+n\}$ such that $\#(\sigma(\{1,\ldots,r\})\cap\{1,\ldots,r\})=u$ and exactly $m$ elements from the set $\{r+1,\ldots,r+n\}$ are fixed points of $\sigma$. In other words,
\begin{align*}
    \cal{D}(r,u,m,n) &:= \big\{\ \sigma\in \cal{S}_{r+n}\ \big|\ \sigma(x)=y \textrm{ for exactly } u \textrm{ pairs } (x,y) \textrm{ such that } 1\leq x,y\leq r \textrm{ and}  \\
     & \hspace{2.65cm} \sigma(t)=t \textrm{ for exactly } m \textrm{ elements } r+1\leq t\leq r+n\ \big\}. 
\end{align*}
Put $d(r,u,m,n)=\#\cal{D}(r,u,m,n)$. Moreover, we define
\begin{itemize}
    \item $\cal{D}_k(r,u,m,n)$ as the subset of $\cal{D}(r,u,m,n)$ containing precisely the permutations such that $1,\ldots ,r$ lie in exactly $k$ disjoint cycles,
    \item $\cal{D}^{(i)}(r,u,m,n)$ as the subset of $\cal{D}(r,u,m,n)$ containing precisely the permutations of parity equal to $i$,
    \item $\cal{D}_k^{(i)}(r,u,m,n)$ as the subset of $\cal{D}(r,u,m,n)$ containing precisely the permutations of parity equal to $i$ and such that $1,\ldots ,r$ lie in exactly $k$ disjoint cycles.
\end{itemize}
Moreover, the numbers $d_k(r,u,m,n)$, $d^{(i)}(r,u,m,n)$, $d_k^{(i)}(r,u,m,n)$ are the cardinalities of the sets $\cal{D}_k(r,u,m,n)$, $\cal{D}^{(i)}(r,u,m,n)$, and $\cal{D}_k^{(i)}(r,u,m,n)$, respectively.

It follows from the very definition that $d(r,0,0,n)=d(r,n)$ is the number of block derangements. The following identity is a bit less obvious but still easy to prove.

\begin{prop}\label{P1}
    For each $n,m,r\in\N$ we have
    \begin{align*}
        d_{n,m}^{(r)}=\sum_{u=0}^r d(r,u,m,n).
    \end{align*}
\end{prop}

\begin{proof}
    Note that $\sum_{u=0}^r d(r,u,m,n)$ is the number of all the permutations of a set with $n+r$ elements that have exactly $m$ fixed points among $r+1,\ldots r+n$. This is the same as the number of bijections between two $(n+r)$-element sets with $n$-element intersection that have exactly $m$ fixed points. The latter number is $d_{n,m}^{(r)}$ from the definition.
\end{proof}

The above proposition shows that the numbers $d(r,u,m,n)$ carry more detailed information than $d_{n,m}^{(r)}$. Thus, they may be treated as a generalization or refinement of generalized recontres numbers. 

Let $\cal{D}_{r,u,m}(n)$ be the set of permutations $\sigma\in \cal{S}_{r+n}$ such that every two distinct elements of the set $\{1,\ldots,r\}$ are in two disjoint cycles of $\sigma$ and exactly $u$ elements from the set $\{1,\ldots,r\}$ and $m$ elements from the set $\{r+1,\ldots,r+n\}$ are fixed points of $\sigma$. That is,
\begin{align*}
    \cal{D}_{r,u,m}(n) & := \big\{\ \sigma\in \cal{S}_{r+n}\ \big|\ \forall_{1\leq x<y\leq r} \  x \textrm{ and } y \textrm{ are in disjoint cycles of } \sigma, \textrm{ and} \\
     & \hspace{2.65cm}  \sigma(z)=z \textrm{ for exactly } u \textrm{ elements } 1\leq z\leq r, \textrm{ and}  \\
     & \hspace{2.65cm} \sigma(t)=t \textrm{ for exactly } m \textrm{ elements } r+1\leq t\leq r+n\ \big\}.
\end{align*}

By $\cal{D}_{r,u,m}^{(i)}(n)$ we mean the subset of $\cal{D}_{r,u,m}(n)$ containing permutations of parity equal to $i$. Put $d_{r,u,m}(n)=\#\cal{D}_{r,u,m}(n)$ and $d_{r,u,m}^{(i)}(n)=\#\cal{D}_{r,u,m}^{(i)}(n)$. If $m=u=0$, then we recover $r$-derangements and thus write $\cal{D}_r(n)$, $d_r(n)$, $\cal{D}_r^{(i)}(n)$ and $d_r^{(i)}(n)$ instead of $\cal{D}_{r,0,0}(n)$, $d_{r,0,0}(n)$, $\cal{D}_{r,0,0}^{(i)}(n)$ and $d_{r,0,0}^{(i)}(n)$, respectively. If additionally $r=0$, then we will write $\cal{D}(n)$, $d(n)$, $\cal{D}^{(i)}(n)$ and $d^{(i)}(n)$, for short.

In the sequel, we will use unsigned Stirling numbers of the first kind, denoted by $\left[\begin{array}{c}r\\k\end{array}\right]$ for natural numbers $r$ and $k$. They are defined as the coefficients of rising factorials:
\begin{align*}
    (x)^{(r)}:=\prod_{j=0}^{r-1}(x+j)=\sum_{k=0}^{r}\left[\begin{array}{c}r\\k\end{array}\right]x^{k}.
\end{align*}
Equivalently, $\left[\begin{array}{c}r\\k\end{array}\right]$ is equal to the number of permutations of the set $\{1,\ldots ,r\}$ with exactly $k$ disjoint cycles. In particular,
\begin{align}\label{eqStirling}
    \sum_{k=0}^{r}\left[\begin{array}{c}r\\k\end{array}\right]=r!,
\end{align}
because this is the number of all permutations of $\{1,\ldots ,r\}$.

For an illustration, let us list all the elements of the sets $\cal{D}(2,1,1,3)$ and $\cal{D}_{2,1,1}(3)$, respectively. All the permutations below are written in the form of products of pairwise disjoint cycles.
\begin{align*}
    \mathcal{D}(2,1,1,3) & = \left\{\begin{array}{cccccc}
        (234), & (235), & (243), & (245), & (253), & (254), \\
        (134), & (135), & (143), & (145), & (153), & (154), \\
        (1234), & (1235), & (1243), & (1245), & (1253), & (1254), \\
        (2134), & (2135), & (2143), & (2145), & (2153), & (2154)
    \end{array}\right\}, \\
    \mathcal{D}(2,1,1,3) & = \left\{\begin{array}{cccccc}
        (234), & (235), & (243), & (245), & (253), & (254), \\
        (134), & (135), & (143), & (145), & (153), & (154)
    \end{array}\right\}.
\end{align*}

We can see above that the set $\cal{D}(2,1,1,3)$ contains twice as many elements as $\cal{D}_{2,1,1}(3)$ (that is, $24$ elements vs $12$ elements). As it turns out, in general the cardinality of $\cal{D}(r,u,m,n)$ is $r!$ times greater than the cardinality of $\cal{D}_{r,u,m}(n)$ (see Theorem \ref{main}\eqref{EqMain2}).

We have used PARI/GP code (written for us by Pavlo Yatsyna) to calculate the numbers $d(r,u,m,n)$ and $d_{r,u,m}(n)$ for small cases. Specifically, we chose $r\in\{2,3\}$, $u,m\in\{0,1,2\}$ and $n\in\{0,1,2,3,4,5,6,7\}$. In fact, we managed to do that in two ways: first, by computing the sets $\mathcal{D}(r,u,m,n)$ and $\mathcal{D}_{r,u,m}(n)$ directly; next, by using the explicit formulae
\begin{align*}
    d_{r,u,m}(n)=\binom{r}{u}\binom{n}{m}\sum_{j=r-u}^{n-m}\binom{j}{r-u}\frac{(n-m)!}{(n-m-j)!}(-1)^{n-m-j}
\end{align*}
and $d(r,u,m,n)=r!d_{r,u,m}(n)$, which are shown in our Theorem \ref{main}\eqref{EqMain2} and Corollary \ref{exactd_rum(n)}. Therefore, we have verified the result numerically in the presented range.

The code can be found in Appendix A.

\begin{center}
    \begin{longtable}{|c|c|c||c|c|c|}
        \hline
         $(r,u,m,n)$ & $d(r,u,m,n)$ & $d_{r,u,m}(n)$ & $(r,u,m,n)$ & $d(r,u,m,n)$ & $d_{r,u,m}(n)$ \\ 
         \hline
         $(2,0,0,0)$ & $0$ & $0$ & $(3,0,0,0)$ & $0$ & $0$ \\
         \hline
         $(2,0,0,1)$ & $0$ & $0$ & $(3,0,0,1)$ & $0$ & $0$ \\
         \hline
         $(2,0,0,2)$ & $4$ & $2$ & $(3,0,0,2)$ & $0$ & $0$ \\
         \hline
         $(2,0,0,3)$ & $24$ & $12$ & $(3,0,0,3)$ & $36$ & $6$ \\
         \hline
         $(2,0,0,4)$ & $168$ & $84$ & $(3,0,0,4)$ & $432$ & $72$ \\
         \hline
         $(2,0,0,5)$ & $1280$ & $640$ & $(3,0,0,5)$ & $4680$ & $780$ \\
         \hline
         $(2,0,0,6)$ & $10860$ & $5430$ & $(3,0,0,6)$ & $51120$ & $8520$ \\
         \hline
         $(2,0,0,7)$ & $101976$ & $50988$ & $(3,0,0,7)$ & $585900$ & $97650$ \\
         \hline
         $(2,0,1,0)$ & $0$ & $0$ & $(3,0,1,0)$ & $0$ & $0$ \\
         \hline
         $(2,0,1,1)$ & $0$ & $0$ & $(3,0,1,1)$ & $0$ & $0$ \\
         \hline
         $(2,0,1,2)$ & $0$ & $0$ & $(3,0,1,2)$ & $0$ & $0$ \\
         \hline
         $(2,0,1,3)$ & $12$ & $6$ & $(3,0,1,3)$ & $0$ & $0$ \\
         \hline
         $(2,0,1,4)$ & $96$ & $48$ & $(3,0,1,4)$ & $144$ & $24$ \\
         \hline
         $(2,0,1,5)$ & $840$ & $420$ & $(3,0,1,5)$ & $2160$ & $360$ \\
         \hline
         $(2,0,1,6)$ & $7680$ & $3840$ & $(3,0,1,6)$ & $28080$ & $4680$ \\
         \hline
         $(2,0,1,7)$ & $76020$ & $28010$ & $(3,0,1,7)$ & $357840$ & $59640$ \\
         \hline
         $(2,0,2,0)$ & $0$ & $0$ & $(3,0,2,0)$ & $0$ & $0$ \\
         \hline
         $(2,0,2,1)$ & $0$ & $0$ & $(3,0,2,1)$ & $0$ & $0$ \\
         \hline
         $(2,0,2,2)$ & $0$ & $0$ & $(3,0,2,2)$ & $0$ & $0$ \\
         \hline
         $(2,0,2,3)$ & $0$ & $0$ & $(3,0,2,3)$ & $0$ & $0$ \\
         \hline
         $(2,0,2,4)$ & $24$ & $12$ & $(3,0,2,4)$ & $0$ & $0$ \\
         \hline
         $(2,0,2,5)$ & $240$ & $120$ & $(3,0,2,5)$ & $360$ & $60$ \\
         \hline
         $(2,0,2,6)$ & $2520$ & $1260$ & $(3,0,2,6)$ & $6480$ & $1080$ \\
         \hline
         $(2,0,2,7)$ & $26880$ & $13440$ & $(3,0,2,7)$ & $98280$ & $16380$ \\
         \hline
         $(2,1,0,0)$ & $0$ & $0$ & $(3,1,0,0)$ & $0$ & $0$ \\
         \hline
         $(2,1,0,1)$ & $4$ & $2$ & $(3,1,0,1)$ & $0$ & $0$ \\
         \hline
         $(2,1,0,2)$ & $8$ & $4$ & $(3,1,0,2)$ & $36$ & $6$ \\
         \hline
         $(2,1,0,3)$ & $36$ & $18$ & $(3,1,0,3)$ & $216$ & $36$ \\
         \hline
         $(2,1,0,4)$ & $176$ & $88$ & $(3,1,0,4)$ & $1512$ & $252$ \\
         \hline
         $(2,1,0,5)$ & $1060$ & $530$ & $(3,1,0,5)$ & $11520$ & $1920$ \\
         \hline
         $(2,1,0,6)$ & $7416$ & $3708$ & $(3,1,0,6)$ & $97740$ & $16290$ \\
         \hline
         $(2,1,0,7)$ & $59332$ & $29666$ & $(3,1,0,7)$ & $917784$ & $152964$ \\
         \hline
         $(2,1,1,0)$ & $0$ & $0$ & $(3,1,1,0)$ & $0$ & $0$ \\
         \hline
         $(2,1,1,1)$ & $0$ & $0$ & $(3,1,1,1)$ & $0$ & $0$ \\
         \hline
         $(2,1,1,2)$ & $8$ & $4$ & $(3,1,1,2)$ & $0$ & $0$ \\
         \hline
         $(2,1,1,3)$ & $24$ & $12$ & $(3,1,1,3)$ & $108$ & $18$ \\
         \hline
         $(2,1,1,4)$ & $144$ & $72$ & $(3,1,1,4)$ & $864$ & $144$ \\
         \hline
         $(2,1,1,5)$ & $880$ & $440$ & $(3,1,1,5)$ & $7560$ & $1260$ \\
         \hline
         $(2,1,1,6)$ & $6360$ & $3180$ & $(3,1,1,6)$ & $69120$ & $11520$ \\
         \hline
         $(2,1,1,7)$ & $51912$ & $25956$ & $(3,1,1,7)$ & $684180$ & $114030$ \\
         \hline
         $(2,1,2,0)$ & $0$ & $0$ & $(3,1,2,0)$ & $0$ & $0$ \\
         \hline
         $(2,1,2,1)$ & $0$ & $0$ & $(3,1,2,1)$ & $0$ & $0$ \\
         \hline
         $(2,1,2,2)$ & $0$ & $0$ & $(3,1,2,2)$ & $0$ & $0$ \\
         \hline
         $(2,1,2,3)$ & $12$ & $6$ & $(3,1,2,3)$ & $0$ & $0$ \\
         \hline
         $(2,1,2,4)$ & $48$ & $24$ & $(3,1,2,4)$ & $216$ & $36$ \\
         \hline
         $(2,1,2,5)$ & $360$ & $180$ & $(3,1,2,5)$ & $2160$ & $360$ \\
         \hline
         $(2,1,2,6)$ & $2640$ & $1320$ & $(3,1,2,6)$ & $22680$ & $3780$ \\
         \hline
         $(2,1,2,7)$ & $22260$ & $11130$ & $(3,1,2,7)$ & $241920$ & $40320$ \\
         \hline
    \end{longtable}
\end{center}

\section{Main Lemma}\label{SecMainLemma}

Let us fix numbers $r,u,m,n\in\mathbb{N}$ and a subset $\cal{A}\subseteq \cal{D}(r,u,m,n)$. Let us consider a function $\TA:\cal{A}\to \cal{D}_{r,u,m}(n)$ constructed in the following way. If $\sigma\in\cal{A}$, then we write $\sigma$ as a product of pairwise disjoint cycles and begin each cycle from its least element. Next, every cycle in $\sigma$ having more than one element from the set $\{1,\ldots,r\}$ splits into several cycles as follows. Denote all the elements of the set $\{1,\ldots,r\}$ contained in the cycle by $a_1,\ldots,a_s$ and denote the cycle in the form
$$
(a_1 b_{1,1} \ldots b_{1,t_1} a_2 b_{2,1} \ldots b_{2,t_2} \ldots a_s b_{s,1} \ldots b_{s,t_s}),
$$ 
where $s>1$ and $t_1,\ldots,t_s\in\N$. Next, we split the cycle: 
$$
(a_1 b_{1,1} \ldots b_{1,t_1} a_2 b_{2,1} \ldots b_{2,t_2} \ldots a_s b_{s,1} \ldots b_{s,t_s})\mapsto (a_1 b_{1,1} \ldots b_{1,t_1})(a_2 b_{2,1} \ldots b_{2,t_2})\cdots(a_s b_{s,1} \ldots b_{s,t_s}).
$$ 
Finally, our permutation $\TA(\sigma)$ is the permutation obtained from $\sigma$ by splitting all its cycles containing more than one element from the set $\{1,\ldots,r\}$ as above.

The following fact will be crucial for our argument.

\begin{lem}\label{LemMain}
Let $\rho\in\cal{D}_{r,u,m}(n)$. Let $\cal{A}\subseteq \cal{D}(r,u,m,n)$ be such that for every $\sigma\in \cal{A}$ the number of disjoint cycles of $\sigma$ containing numbers from $\{1,\ldots ,r\}$ satisfies a fixed condition $w$ (i.e., the condition $w$ may mean for example the conditions like `having some fixed value' or `laying in some prescribed residue class modulo some positive integer'; $w$ may depend on $\rho$). Then
\begin{align*}
    \# \TA^{-1}(\rho )= r!\sum_{\substack{j_1,\ldots,j_r\in\N \\j_1+2j_2+\cdots+rj_r=r \\ j_{1}+\cdots +j_{r}\ {\rm satisfies}\ w}}\frac{1}{\prod_{i=1}^r i^{j_i}j_i!}.
\end{align*}
\end{lem}
\begin{proof}
We can create $\sigma\in\cal{D}(r,u,m,n)$ such that $\TA(\sigma)=\rho$ by gluing some cycles of $\rho$ containing elements from the set $\{1,\ldots,r\}$ according to the recipe: 
$$
(a_1 b_{1,1} \ldots b_{1,t_1})(a_2 b_{2,1} \ldots b_{2,t_2})\cdots(a_s b_{s,1} \ldots b_{s,t_s})\mapsto (a_1 b_{1,1} \ldots b_{1,t_1} a_2 b_{2,1} \ldots b_{2,t_2} \ldots a_s b_{s,1} \ldots b_{s,t_s}).
$$
Let us compute in how many ways we can create $\sigma$ from $\rho$ as above.

We have $r$ pairwise disjoint cycles in $\rho$ containing elements of the set $\{1,\ldots,r\}$, as $\rho\in\cal{D}_{r,u,m}(n)$. Assume that we intend to partition these cycles into $j_1$ subsets of cardinality $1$, $j_2$ subsets of cardinality $2$, $j_3$ subsets of cardinality $3, \ldots, j_r$ subsets of cardinality $r$. We can do it in $\frac{r!}{(1!)^{j_1}j_1!\cdot (2!)^{j_2}j_2!\cdot (3!)^{j_3}j_3!\cdots (r!)^{j_r}j_r!}$ ways. In each subset of $i$ cycles we take as the first the cycle with the least element from the set $\{1,\ldots,r\}$ and order the remaining cycles in $(i-1)!$ ways. Next, we glue the cycles in this order. Let us notice that fixing the cycle with the least element as the first one in the ordering makes it so each cycle of $\sigma$ can be attained by gluing from exactly one ordering. Notice that the assumption that $\sigma\in \cal{A}$ implies that $j_{1}+\cdots +j_{r}$ has to satisfy $w$. Summing up, the number of ways of creating $\sigma$ from $\rho$ is equal to 
$$
\sum_{\substack{j_1,\ldots,j_r\in\N \\j_1+2j_2+\cdots+rj_r=r \\ j_{1}+\cdots +j_{r} \textrm{ satisfies } w}}\frac{r!}{\prod_{i=1}^r (i!)^{j_i}j_i!}\cdot\prod_{i=1}^r ((i-1)!)^{j_i}=r!\sum_{\substack{j_1,\ldots,j_r\in\N \\j_1+2j_2+\cdots+rj_r=r \\ j_{1}+\cdots +j_{r} \textrm{ satisfies } w}}\frac{1}{\prod_{i=1}^r i^{j_i}j_i!}.
$$
\end{proof}

\section{Main result}\label{SecMainResults}

Our main result is the following.

\begin{thm}\label{main}
Let us fix $r,u,m,n,k\in\N$ and $i\in\{0,1\}$. Then, we have
\begin{align}\label{geneq}
d_{k}(r,u,m,n)=\left[\begin{array}{c}r\\k\end{array}\right]d_{r,u,m}(n).
\end{align}

In particular,
\begin{align}\label{EqMain2}
d(r,u,m,n)=r!d_{r,u,m}(n).
\end{align}

Moreover,
\begin{align}\label{EqMain3}
d^{(i)}(r,u,m,n)=
\begin{cases}
\frac{r!}{2}d_{r,u,m}(n), &\mbox{ if } r\geq 2,\\    
d_{r,u,m}^{(i)}(n), &\mbox{ if } r\in\{0,1\},
\end{cases}
\end{align}
and
\begin{align}\label{EqMain4}
d_{k}^{(i)}(r,u,m,n)=\left[\begin{array}{c}r\\k\end{array}\right]d^{((r+k+i)\bmod{2})}_{r,u,m}(n).
\end{align}
\end{thm}

\begin{proof}

Lemma \ref{LemMain} applied to $\cal{A}=\mathcal{A}_{1}:=\cal{D}_{k}(r,u,m,n)$ (that is, with the condition $w:=$ `be equal to $k$') implies
\begin{align*}
    \#\Ta{1}^{-1}(\rho)=r!\sum_{\substack{j_1,\ldots,j_r\in\N \\j_1+2j_2+\cdots+rj_r=r \\ j_{1}+\cdots +j_{r}=k}}\frac{1}{\prod_{i=1}^r i^{j_i}j_i!}=:f(r,k)
\end{align*}
for any $\rho \in\cal{D}_{r,u,m}(n)$. On the other hand,
$$
d_{k}(r,u,m,n)=\sum_{\rho\in\cal{D}_{r,u,m}(n)}\# \Ta{1}^{-1}(\rho)=\sum_{\rho\in\cal{D}_{r,u,m}(n)}f(r,k)=f(r,k)d_{r,u,m}(n).
$$ 
Observe that $f(r,k)$ depends exclusively on $r$ and $k$. Hence, substituting $u=r$ and $m=n=0$ we get $f(r,k)=f(r,k)d_{r,r,0}(0)=d_{k}(r,r,0,0)=\left[\begin{array}{c}r\\k\end{array}\right]$. This concludes the proof of the equality $d_{k}(r,u,m,n)=\left[\begin{array}{c}r\\k\end{array}\right]d_{r,u,m}(n)$ for all $r,u,m,n,k\in\N$.

In order to prove (\ref{EqMain2}) it is enough to sum both sides of (\ref{geneq}) over $k\in\{0\ldots ,r\}$. Indeed, we get
\begin{align*}
    d(r,u,m,n)=\sum_{k=0}^{r}d_{k}(r,u,m,n)=\left(\sum_{k=0}^{r}\left[\begin{array}{c}r\\k\end{array}\right]\right)d_{r,u,m}(n)=r!d_{r,u,m}(n).
\end{align*}
The last equality follows from (\ref{eqStirling}).

Let us move to the proof of (\ref{EqMain3}). For any permutation $\sigma$ of the set $\{1,\ldots ,r+n\}$ let $k_{\sigma}$ denote the number of disjoint cycles in $\sigma$ and $k'_{\sigma}$ denote the number of those cycles of $\sigma$ which contain at least one element of $\{1,\ldots ,r\}$. For any permutation $\rho\in\cal{D}_{r,u,m}(n)$ and $i\in\{0,1\}$ let $\mathcal{A}_{\rho,i}$ be the set of all permutations $\sigma\in\mathcal{D}(r,u,m,n)$ such that
\begin{align}\label{eqLength}
    k_{\sigma}'\equiv r+i+\Par\rho \pmod{2}.
\end{align}

Suppose that $\sigma\in\T_{\mathcal{D}(r,u,m,n)}^{-1}(\rho)$. Then, using the fact that $k_{\rho}'=r$ (because $\rho\in\mathcal{D}_{r,u,m}(n)$) and the well-known property
\begin{align*}
    \Par \sigma \equiv k_{\sigma}+r+n\pmod 2,
\end{align*}
we obtain that
\begin{align}\label{equivlong}
    \Par\rho \equiv k_{\rho}+r+n=k_{\rho}-k_{\sigma} + \big(k_{\sigma} +r+n\big) \equiv k_{\rho}-k_{\sigma}+\Par\sigma =r-k'_{\sigma}+\Par\sigma \pmod{2},
\end{align}
where the last equality follows from the fact that $\sigma$ and $\rho$ have exactly the same cycles that do not contain any of the numbers $1,\ldots ,r$. Therefore, the condition \eqref{eqLength} is equivalent to $\Par\sigma=i$. Indeed, it follows from \eqref{equivlong} that 
\begin{align}\label{equivParSigma}
    \Par\sigma\equiv\Par\rho+r+k'_{\sigma}\pmod{2}.
\end{align}
So, if \eqref{eqLength} holds, then inserting $r+i+\Par\rho$ in the place of $k'_{\sigma}$ in \eqref{equivParSigma} gives us the equality $\Par\sigma=i$. On the other hand, from \eqref{equivlong} we conclude that
\begin{align*}
    k'_{\sigma}\equiv\Par\rho+\Par\sigma+r\pmod{2}.
\end{align*}
Then, if $\Par\sigma=i$, we get \eqref{eqLength}. Summing up, for every $\rho\in\mathcal{D}_{r,u,m}(n)$ we have $\mathcal{A}_{\rho,i}= \mathcal{D}^{(i)}(r,u,m,n)$.




We now apply Lemma \ref{LemMain} to the subset $\cal{A}=\cal{A}_{\rho,i}$
. Hence, the number of ways of creating a partition $\sigma$ with $\Par \sigma = i$ from $\rho$ is equal to
\begin{align*}
    \# \Ta{\rho,i}^{-1}(\rho) & =r!\sum_{\substack{j_1,\ldots,j_r\in\N \\j_1+2j_2+\ldots+rj_r=r \\ j_{1}+\cdots +j_{r}\equiv r+i+\Par\rho\Mod 2}}\frac{1}{\prod_{i=1}^r i^{j_i}j_i!}.
\end{align*}

For $\varepsilon\in\{0,1\}$ let 
\begin{align*}
    f_{\varepsilon}(r):=r!\sum_{\substack{j_1,\ldots,j_r\in\N \\j_1+2j_2+\cdots+rj_r=r \\ j_{1}+\cdots +j_{r}\equiv \varepsilon\Mod 2}}\frac{1}{\prod_{i=1}^r i^{j_i}j_i!}.
\end{align*}

Then,
\begin{align*}
    d^{(i)}(r,u,m,n) &= \sum_{\rho\in \cal{D}_{r,u,m}(n)}\#\T_{\cal{D}^{(i)}(r,u,m,n)}^{-1}(\rho)=\sum_{\rho\in \cal{D}_{r,u,m}(n)}\#\T_{\cal{A}_{\rho,i}}^{-1}(\rho) \\
    &= \sum_{\rho\in \cal{D}_{r,u,m}(n)}r!\sum_{\substack{j_1,\ldots,j_r\in\N \\j_1+2j_2+\cdots+rj_r=r \\ j_{1}+\cdots +j_{r}\equiv r+i+\Par\rho\Mod 2}}\frac{1}{\prod_{i=1}^r i^{j_i}j_i!} \\
    &= f_{0}(r)\cdot\#\{ \rho\in \cal{D}_{r,u,m}(n)\ |\ \Par\rho\equiv r+i\Mod 2  \} \\
    & \hspace{1cm} + f_{1}(r)\cdot\#\{ \rho\in \cal{D}_{r,u,m}(n)\ |\ \Par\rho\equiv r+i+1\Mod 2  \}\\
    &= f_0(r)d_{r,u,m}^{((r+i)\bmod{2})}(n)+f_1(r)d_{r,u,m}^{((r+i+1)\bmod{2})}(n).
\end{align*}

Let us consider two cases. First, suppose that $i\equiv r\pmod{2}$. Because both numbers $f_{0}(r)$ and $f_{1}(r)$ depend only on $r$, if we put $u=r$ and $m=n=0$, we immediately get
\begin{align*}
f_{0}(r)=d^{(r\bmod 2)}(r,r,0,0)=
\begin{cases}
\frac{r!}{2}, &\mbox{ if } r\geq 2,\\
1-r, &\mbox{ if } r\in\{0,1\}. 
\end{cases}
\end{align*}
On the other hand, observe that by (\ref{EqMain2}) we have $f_{0}(r)+f_{1}(r)=r!$. Hence, for $r\geq 2$ we get $f_{0}(r)=f_{1}(r)=\frac{r!}{2}$, and finally
\begin{align*}
    d^{(i)}(r,u,m,n)= &\ \frac{r!}{2}\big(d_{r,u,m}^{((r+i)\bmod{2})}(n)+d_{r,u,m}^{((r+i+1)\bmod{2})}(n)\big)=\frac{r!}{2}d_{r,u,m}(n).
\end{align*}
If $r\in\{0,1\}$, then $f_{0}(r)=1-r$ and $f_{1}(r)=r$. Thus,
\begin{align*}
    d^{(i)}(r,u,m,n)= &\ (1-r)d_{r,u,m}^{((r+i)\bmod{2})}(n)+rd_{r,u,m}^{((r+i+1)\bmod{2})}(n)=d_{r,u,m}^{(i)}(n),
\end{align*}
because we assume $i\in\{0,1\}$.

In the second case, that is $i\equiv r+1\pmod{2}$, we proceed analogously and obtain exactly the same formulae.

For the proof of \eqref{EqMain4} we apply Lemma \ref{LemMain} to the subset $\cal{A}=\mathcal{A}_{\rho,i,k}$ of $\cal{D}_k^{(i)}(r,u,m,n)$ containing all the permutations $\sigma$ such that the number of the cycles of $\sigma$ which contain at least one element of $\{1,\ldots ,r\}$ is equal to $k$ and congruent to $r+i+\Par\rho$ modulo $2$ (that is, with the condition $w:=$ `be equal to $k$ and equivalent to $r+i+\Par\rho$ modulo $2$'). Hence, the number of ways of creating $\sigma\in\cal{D}_k^{(i)}(r,u,m,n)$ from $\rho$ is equal to
\begin{align*}
    \#\T_{\cal{D}_k^{(i)}(r,u,m,n)}^{-1}(\rho) & =\# \Ta{\rho,i,k}^{-1}(\rho) = r!\sum_{\substack{j_1,\ldots,j_r\in\N \\j_1+2j_2+\ldots+rj_r=r \\ j_{1}+\cdots +j_{r}=k \\ j_{1}+\cdots +j_{r}\equiv r+i+\Par\rho\Mod 2}}\frac{1}{\prod_{i=1}^r i^{j_i}j_i!} \\
    &=
    \begin{cases}
    f(r,k), & \mbox{ if }\Par\rho\equiv r+k+i\Mod{2}\\
    0, & \mbox{ if }\Par\rho\equiv r+k+i+1\Mod{2}
    \end{cases}.
\end{align*}
where 
$f(r,k)=
\left[\begin{array}{c}
     r  \\
     k 
\end{array}\right]$. Hence,
\begin{align*}
   d_k^{(i)}(r,u,m,n) &=\sum_{\rho\in\cal{D}_{r,u,m}(n)}\T_{\cal{D}_k^{(i)}(r,u,m,n)}^{-1}(\rho) = \sum_{\rho\in\cal{D}_{r,u,m}(n)}\Ta{\rho,i,k}^{-1}(\rho) \\
    &= f(r,k)\cdot\#\{ \rho\in \cal{D}_{r,u,m}(n)\ |\ \Par\rho\equiv r+k+i\Mod 2  \} \\ 
    &= \left[\begin{array}{c}
     r  \\
     k 
\end{array}\right]
d_{r,u,m}^{((r+k+i)\bmod{2})}(n).
\end{align*}
The proof is complete.
\end{proof}


\section{Approach via Power Series}\label{SecPowerSeries}

It is worth noting that the values $f(r,k)$ from the proof of Theorem \ref{main} can be computed with the use of power series. In fact, the following is true.

\begin{prop}\label{gen}
For each $r\in\N$ we have
\begin{align}
\sum_{\substack{j_1,\ldots,j_r\in\N \\j_1+2j_2+\cdots+rj_r=r}}\left(\prod_{k=1}^r k^{j_k}j_k!\right)^{-1}= &\ 1, \label{sum} \\
\sum_{\substack{j_1,\ldots,j_r\in\N \\j_1+2j_2+\cdots+rj_r=r\\2\mid j_1+\cdots+j_r}}\left(\prod_{k=1}^r k^{j_k}j_k!\right)^{-1}= &\ \begin{cases}
1-r&\mbox{ for }r\in\{0,1\}\\
\frac{1}{2}&\mbox{ for }r\geq 2
\end{cases}, \label{sumeven} \\
\sum_{\substack{j_1,\ldots,j_r\in\N \\j_1+2j_2+\cdots+rj_r=r\\2\nmid j_1+\cdots+j_r}}\left(\prod_{k=1}^r k^{j_k}j_k!\right)^{-1}= &\ \begin{cases}
r&\mbox{ for }r\in\{0,1\}\\
\frac{1}{2}&\mbox{ for }r\geq 2
\end{cases}.\label{sumodd}
\end{align}
In general,
\begin{align}\label{sumrk}
\sum_{\substack{j_1,\ldots,j_r\in\N \\j_1+2j_2+\cdots+rj_r=r\\j_1+\cdots+j_r=k}}\left(\prod_{k=1}^r k^{j_k}j_k!\right)^{-1}=\frac{\left[\begin{array}{c}r\\k\end{array}\right]}{r!}
\end{align}
for each $r,k\in\N$.

In particular, if $d\in\N_+$ and $c\in\{0,\ldots,d-1\}$, then
\begin{align}\label{sumrest}
\sum_{\substack{j_1,\ldots,j_r\in\N \\j_1+2j_2+\cdots+rj_r=r\\j_1+\cdots+j_r\equiv c\Mod{d}}}\left(\prod_{k=1}^r k^{j_k}j_k!\right)^{-1}=\frac{1}{r!}\sum_{\substack{1\leq k\leq r\\k\equiv c\Mod{d}}}\left[\begin{array}{c}r\\k\end{array}\right].
\end{align}
\end{prop}

\begin{proof}
Let 
$$
\N^{(\N_+)}=\{(j_l)_{l\in\N_+}\ |\ \forall_{l\in\N_+}\ j_l\in\N\ \textrm{ and }\  j_l=0\ \textrm{ for all but finitely many $l$} \}.
$$
Consider the power series 
$$
F(x,y)=\sum_{r=0}^{+\infty}\sum_{k=0}^{+\infty}\sum_{\substack{j_1,\ldots,j_r\in\N \\j_1+2j_2+\cdots+rj_r=r\\j_1+\cdots+j_r=k}}\left(\prod_{l=1}^r l^{j_l}j_l!\right)^{-1}x^ry^k.
$$
Compute:
\begin{align*}
F(x,y)=&\ \sum_{r=0}^{+\infty}\sum_{k=0}^{+\infty}\sum_{\substack{j_1,\ldots,j_r\in\N \\j_1+2j_2+\cdots+rj_r=r\\j_1+\cdots+j_r=k}}\left(\prod_{l=1}^r l^{j_l}j_l!\right)^{-1}x^ry^k
=\sum_{(j_l)_{l\in\N_+}\in\N^{(\N_+)}}\frac{x^{\sum_{l=1}^{+\infty}lj_l}y^{\sum_{l=1}^{+\infty}j_l}}{\prod_{l=1}^{+\infty} l^{j_l}j_l!} \\
= & \sum_{(j_l)_{l\in\N_+}\in\N^{(\N_+)}}\prod_{l=1}^{+\infty}\frac{x^{lj_l}y^{j_l}}{l^{j_l}j_l!}=\prod_{l=1}^{+\infty}\sum_{j_l=0}^{+\infty}\frac{1}{j_l!}\left(\frac{x^{l}y}{l}\right)^{j_l}
=\prod_{l=1}^{+\infty}e^{\frac{x^{l}y}{l}} \\
= &\ e^{\sum_{l=1}^{+\infty}\frac{x^{l}y}{l}}=e^{y\log\frac{1}{1-x}}=\left(\frac{1}{1-x}\right)^y.
\end{align*}
Note that 
$$
\sum_{\substack{j_1,\ldots,j_r\in\N \\j_1+2j_2+\cdots+rj_r=r\\2\mid j_1+\cdots+j_r}}\left(\prod_{k=1}^r k^{j_k}j_k!\right)^{-1}x^r=\frac{1}{2}(F(x,1)+F(x,-1))=\frac{1}{2}\left(\frac{1}{1-x}+1-x\right)=1+\frac{1}{2}\sum_{r=2}^{+\infty}x^r.
$$
Comparing the coefficients, we get the equality \eqref{sumeven} for each $r\in\N$. Analogously, 
$$
\sum_{\substack{j_1,\ldots,j_r\in\N \\ j_1+2j_2+\cdots+rj_r=r \\ 2\nmid j_1+\cdots+j_r}}\left(\prod_{k=1}^r k^{j_k}j_k!\right)^{-1}x^r=\frac{1}{2}(F(x,1)-F(x,-1))=\frac{1}{2}\left(\frac{1}{1-x}-(1-x)\right)=x+\frac{1}{2}\sum_{r=2}^{+\infty}x^r,
$$ 
which proves \eqref{sumodd}. At last, \eqref{sum} follows easily from summing \eqref{sumeven} and \eqref{sumodd} or from the consideration of the power series $F(x,1)$.

The identity \eqref{sumrk} follows from the fact that
\begin{align*}
F(x,y) &= \left(\frac{1}{1-x}\right)^y=(1-x)^{-y}=\sum_{r=0}^{+\infty}\binom{-y}{r}(-x)^r \\
&= \sum_{r=0}^{+\infty}\frac{(-y)_{(r)}}{r!}(-x)^r=\sum_{r=0}^{+\infty}\frac{(y)^{(r)}}{r!}x^r=\sum_{r=0}^{+\infty}\sum_{k=1}^r\frac{\left[\begin{array}{c}r\\k\end{array}\right]}{r!}x^ry^k,
\end{align*}
where $(x)_{(r)}=\prod_{j=0}^{r-1}(x-j)$ and $(x)^{(r)}=\prod_{j=0}^{r-1}(x+j)$ are falling and rising factorials, respectively. Then, \eqref{sumrest} is a direct consequence of \eqref{sumrk}.
\end{proof}

\section{Reduction to the numbers of $r$-derangements}\label{SecDerang}

The following result shows that studying arithmetic properties of the numbers $d_{r,u,m}(n)$ and\linebreak $d_{r,u,m}^{(i)}(n)$ comes down to studying corresponding ones for the numbers $d_r(n)$ and $d_r^{(i)}(n)$.

\begin{prop}\label{byrderangements}
For each $r,u,m,n\in\N$ and $i\in\{0,1\}$ we have
\begin{align*}
   d_{r,u,m}(n)=\binom{r}{u}\binom{n}{m}d_{r-u}(n-m)
\end{align*}
and
\begin{align*}
   d_{r,u,m}^{(i)}(n)=\binom{r}{u}\binom{n}{m}d_{r-u}^{(i)}(n-m).
\end{align*}
\end{prop}

\begin{proof}
Let $\sigma\in\cal{D}_{r,u,m}(n)$. After choosing $u$ fixed points of $\sigma$ from the set $\{1,\ldots ,r\}$ and $m$ fixed points of $\sigma$ from the set $\{r+1,\ldots ,r+n\}$ we can treat $\sigma$ as a fixed-point-free permutation of a set with $(r-u)+(n-m)$ elements. Note that omitting fixed points of $\sigma$ does not affect its parity.
\end{proof}

With the use of the above proposition and an exact formula for the $r$-derangement numbers \cite[Theorem 4, equation (9)]{WMM} we immediately obtain an exact formula for the numbers $d_{r,u,m}(n)$.

\begin{cor}\label{exactd_rum(n)}
For each $r,u,m,n\in\N$ and we have
\begin{align*}
    d_{r,u,m}(n)=\binom{r}{u}\binom{n}{m}\sum_{j=r-u}^{n-m}\binom{j}{r-u}\frac{(n-m)!}{(n-m-j)!}(-1)^{n-m-j}.
\end{align*}
\end{cor}

At this moment our main result can be rewritten in the following way.

\begin{cor}
Let $r,u,m,n,k\in\N$ and $i\in\{0,1\}$. Then, we have
\begin{align*}
d_{k}(r,u,m,n)=\left[\begin{array}{c}r\\k\end{array}\right]\binom{r}{u}\binom{n}{m}d_{r-u}(n-m).
\end{align*}

In particular,
\begin{align*}
d(r,u,m,n)=r!\binom{r}{u}\binom{n}{m}d_{r-u}(n-m).
\end{align*}

Moreover,
\begin{align*}
d^{(i)}(r,u,m,n)=
\begin{cases}
\frac{r!}{2}\binom{r}{u}\binom{n}{m}d_{r-u}(n-m), &\mbox{ if } r\geq 2,\\
\binom{r}{u}\binom{n}{m}d_{r-u}^{(i)}(n-m), &\mbox{ if } r\in\{0,1\},
\end{cases}
\end{align*}
and
\begin{align*}
d_{k}^{(i)}(r,u,m,n)=\left[\begin{array}{c}r\\k\end{array}\right]\binom{r}{u}\binom{n}{m}d_{r-u}^{((r+k+i)\bmod{2})}(n-m).
\end{align*}
\end{cor}

\section{Formulae for numbers of even/odd $r$-derangements and their exponential generating functions}

The properties of numbers of $r$-derangements are widely studied in \cite{WMM}. However, the numbers of even and odd $r$-derangements $d_r^{(i)}(n)$, $i\in\{0,1\}$, have not been investigated up to now. This is why we state and prove some basic facts involving them. The first result is an analogue of \cite[Theorem 2]{WMM}.

\begin{thm}\label{rec}
For each $r,n\in\N$ and $i\in\{0,1\}$ we have
\begin{align*}
   d_r^{(i)}(n)=&rd_{r-1}^{(1-i)}(n-1)+(n-1)d_r^{(1-i)}(n-2)+(n+r-1)d_r^{(1-i)}(n-1),
\end{align*}
where $d_r^{(i)}(n)=0$ for $n<r$ and $d_r^{(i)}(r)=\frac{r!}{2}(1+(-1)^{r+i})$.
\end{thm}

\begin{proof}
Let $n\in\N$ and $\sigma\in\cal{D}_r^{(i)}(n)$. Consider the permutation $\rho=(r+n,\ \sigma(r+n))\circ\sigma$, where $\circ$ denotes the usual product of permutations. We know that $\rho$ fixes $r+n$ and has parity $1-i$. We have three cases.
\begin{itemize}
    \item If $\sigma(\sigma(r+n))=r+n$ and $\sigma(r+n)\in\{1\,\ldots r\}$, then $\rho$ also fixes $\sigma(r+n)$ and can be treated as an element of the set $\cal{D}_{r-1}(n-1)$.
    \item If $\sigma(\sigma(r+n))=r+n$ and $\sigma(r+n)\in\{r+1\,\ldots r+n-1\}$, then $\rho$ also fixes $\sigma(r+n)$ and can be treated as an element of the set $\cal{D}_{r}(n-2)$.
    \item If $\sigma(\sigma(r+n))\neq r+n$, then $\rho$ fixes only $r+n$ and can be treated as an element of the set $\cal{D}_{r}(n-1)$.
\end{itemize}
We thus see that $\sigma\in\cal{D}_r^{(i)}(n)$ can be constructed as $(r+n,\ \sigma(r+n))\circ\rho$, where exactly one of the possibilities holds:
\begin{itemize}
    \item $\rho\in\cal{D}_{r-1}(n-1)$ and $\sigma(r+n)$ can be chosen in $r$ ways;
    \item $\rho\in\cal{D}_{r}(n-2)$ and $\sigma(r+n)$ can be chosen in $n-1$ ways;
    \item $\rho\in\cal{D}_{r}(n-1)$ and $\sigma(r+n)$ can be chosen in $r+n-1$ ways.
\end{itemize}
Finally, we conclude that the recurrence from the statement of the theorem holds.

The equality $d_r^{(i)}(n)=0$ for $n<r$ is clear. For the proof of the equality $d_r^{(i)}(r)=\frac{r!}{2}(1+(-1)^{r+i})$ it suffices to notice that the set $\cal{D}_r(r)$ contains only products of $r$ pairwise disjoint transpositions, which are permutations of parity $r\bmod{2}$.
\end{proof}



Our goal is to compute the exponential generating functions of the sequences $(d_r^{(i)}(n))_{n\in\N}$, $r\in\N$, $i\in\{0,1\}$, and to derive exact formulae for the numbers $d_r^{(i)}(n)$. At first, we consider an auxiliary sequence $(c_r(n))_{n\in\N}$ given by the formula $c_r(n)=d_r^{(1)}(n)-d_r^{(0)}(n)$, $n\in\N$. Directly from Theorem \ref{rec} we get the following.

\begin{cor}\label{recaux}
For each $r,n\in\N$ and $i\in\{0,1\}$ we have
\begin{align*}
    c_r(n)=&-rc_{r-1}(n-1)-(n-1)c_r(n-2)-(n+r-1)c_r(n-1),
\end{align*}
where $c_r(n)=0$ for $n<r$ and $c_r(r)=(-1)^{r+1}r!$.
\end{cor}

Now, we give a formula for the exponential generating function for the sequence $(c_r(n))_{n\in\N}$.

\begin{lem}
Let $r\in\N$. Then, the exponential generating function of the sequence $(c_r(n))_{n\in\N}$ is
\begin{align*}
    G_r(x)=\sum_{n=0}^{+\infty}\frac{c_r(n)}{n!}x^n=\frac{(-1)^{r-1}x^re^{-x}}{(1+x)^{r-1}}.
\end{align*}
\end{lem}
\begin{proof}
    We proceed by induction on $r\in\N$. 
    
    For $r=0$ we know that $d_0^{(i)}(n)=\frac{1}{2}[d_0(n)-(-1)^{n+i}(n-1)]$ (see \cite[Corollary 1]{M}). Thus, $c_0(n)=(-1)^n(n-1)$, so $G_0(x)=-(1+x)e^{-x}$.
    
    For $r\geq 1$ use the recurrence from Corollary \ref{recaux} to get the following:
    \begin{align*}
        G_r'(x) &= \sum_{n=1}^{+\infty}\frac{c_r(n)}{(n-1)!}x^{n-1}=-\sum_{n=1}^{+\infty}\frac{rc_{r-1}(n-1)+(n-1)c_r(n-2)+(n+r-1)c_r(n-1)}{(n-1)!}x^{n-1}\\
        &= -rG_{r-1}(x)-xG_r(x)-xG_r'(x)-rG_{r}(x).
    \end{align*}
    Hence, we obtain a differential equation
    $$(1+x)G_r'(x)=-rG_{r-1}(x)-(x+r)G_r(x).$$
    With the use of the induction assumption 
    $$G_{r-1}(x)=\frac{(-1)^{r-2}x^{r-1}e^{-x}}{(1+x)^{r-2}}$$ 
    we easily check that 
    $$G_r(x)=\frac{(-1)^{r-1}x^re^{-x}}{(1+x)^{r-1}}$$ 
    is the solution of the above equation satisfying the boundary condition $G_r(0)=0$.
\end{proof}

With the knowledge of the exponential generating function of the sequence $(c_r(n))_{n\in\N}$ we are able to derive an explicit formula for the number $c_r(n)$.

\begin{cor}\label{crexact}
For each $r\geq 2$ and $n\geq r$ we have
    \begin{align*}
    c_{r}(n) = (-1)^{n-1} \sum_{j=r}^{n}\frac{n!}{(n-j)!}\binom{j-2}{r-2}.
    \end{align*}
\end{cor}
\begin{proof}
    We write the exponential generating function of the sequence $(c_r(n))_{n\in\N}$ as follows:
    \begin{align*}
        G_r(x) &=(-1)^{r-1}x^re^{-x}\cdot\frac{1}{(1+x)^{r-1}}=\sum_{j=0}^{+\infty}\frac{(-1)^{j+r-1}}{j!}x^{j+r}\cdot\sum_{k=0}^{+\infty}\binom{k+r-2}{r-2}(-x)^k\\
        &=\sum_{n=0}^{+\infty}\frac{x^{n+r}}{n!}\cdot (-1)^{n+r-1}\sum_{k=0}^n\binom{n}{k}\frac{(k+r-2)!}{(r-2)!}\\
        &=\sum_{n=0}^{+\infty}\frac{x^{n+r}}{(n+r)!}\cdot (-1)^{n+r-1}\sum_{k=0}^n\frac{(n+r)!}{k!(n-k)!}\cdot\frac{(k+r-2)!}{(r-2)!}\\
        &=\sum_{n=r}^{+\infty}\frac{x^{n}}{n!}\cdot (-1)^{n-1}\sum_{k=0}^{n-r}\frac{n!}{k!(n-r-k)!}\cdot\frac{(k+r-2)!}{(r-2)!}\\
        &=\sum_{n=r}^{+\infty}\frac{x^{n}}{n!}\cdot (-1)^{n-1}\sum_{k=0}^{n-r}\frac{n!}{(n-r-k)!}\cdot\binom{k+r-2}{r-2}\\
        &=\sum_{n=r}^{+\infty}\frac{x^{n}}{n!}\cdot (-1)^{n-1}\sum_{j=r}^{n}\frac{n!}{(n-j)!}\cdot\binom{j-2}{r-2}.
    \end{align*}
    Since $G_r(x)=\sum_{n=0}^{+\infty}\frac{c_r(n)}{n!}x^n$, by comparing the coefficients we get the statement. 
\end{proof}

An immediate consequence of the above formula for $c_r(n)$ is the following.

\begin{cor}
    For every $r\geq 2$ and $n\geq r$ we have 
    \begin{align*}
        (-1)^{n}\left(d_{r}^{(0)}(n)-d_{r}^{(1)}(n)\right)>0.
    \end{align*}
\end{cor}
\begin{proof}
Corollary \ref{crexact} gives
\begin{align*}
   d_{r}^{(0)}(n)-d_{r}^{(1)}(n) = (-1)^{n} \sum_{j=r}^{n}\frac{n!}{(n-j)!}\binom{j-2}{r-2},
\end{align*}
and the result follows.
\end{proof}

We are now ready to give exact formulae for exponential generating functions of the sequences $(d_r^{(i)}(n))_{n\in\N}$, $r\in\N$, $i\in\{0,1\}$, and the numbers $d_r^{(i)}(n)$ themselves.

\begin{thm}
Let $r\in\N$ and $i\in\{0,1\}$. Then, the exponential generating function of the sequence $(d_r^{(i)}(n))_{n\in\N}$ is
\begin{align*}
    F_r^{(i)}(x)=\sum_{n=0}^{+\infty}\frac{d_r^{(i)}(n)}{n!}x^n=\frac{1}{2}\left(\frac{x^re^{-x}}{(1-x)^{r+1}}+(-1)^{r+i}\frac{x^re^{-x}}{(1+x)^{r-1}}\right).
\end{align*}
\end{thm}
\begin{proof}
    The statement follows easily from the fact that $d_r^{(i)}(n)=\frac{1}{2}[d_r(n)-(-1)^ic_r(n)]$ and from the formula for the exponential generating function $F_r(x)=\frac{x^re^{-x}}{(1-x)^{r+1}}$ of the sequence $(d_r(n))_{n\in\N}$ (see \cite[Theorem 3]{WMM}).
\end{proof}

\begin{thm}\label{ThmExForDri}
Let $r,n\in\N$ with $r\geq 2$ and $i\in\{0,1\}$. Then,
\begin{align*}
   d_r^{(i)}(n)=\frac{1}{2}\sum_{j=r}^n\frac{n!}{(n-j)!}\left((-1)^{n-j}\binom{j}{r}+(-1)^{n+i}\binom{j-2}{r-2}\right).
\end{align*}
\end{thm}
\begin{proof}
    The equality follows from Corollary \ref{crexact} and the explicit formula for $d_r(n)$ in \cite[Theorem 4]{WMM}.  
\end{proof}

A straightforward consequence of Theorem \ref{ThmExForDri} and Proposition \ref{byrderangements} is the following exact formula for the numbers $d_{r,u,m}^{(i)}(n)$ under the condition $r-u\geq 2$. In the case $r-u\in\{0,1\}$ we need to use exact formulae for numbers $d^{(i)}(n)$ of classical derangements with prescribed parity \cite[Corollary 1]{M}. Note here that if $r-u=1$, then $d_1^{(i)}(n)=d^{(i)}(n+1)$.




\begin{cor}
    For all $r,u,m,n\in\mathbb{N}$ and $i\in\{0,1\}$ we have the following:
    \begin{enumerate}
        \item if $r\geq u+2$, then
        \begin{align*}
            d_{r,u,m}^{(i)}(n)=\frac{1}{2}\binom{r}{u}\binom{n}{m}\sum_{j=r-u}^{n-m}\frac{(n-m)!}{(n-m-j)!}\left((-1)^{n-m-j}\binom{j}{r-u}+(-1)^{n-m+i}\binom{j-2}{r-u-2}\right),
        \end{align*}
        \item if $r-u\in\{0,1\}$, then
        \begin{align*}
            d_{r,u,m}^{(i)}(n)=\frac{1}{2}\binom{r}{u}\binom{n}{m}\sum_{j=0}^{n-m-1}\left(\frac{(n-m+r-u)!}{j!}(-1)^j-(-1)^{n-m+r-u+i}(n-m+r-u)\right).
        \end{align*}
    \end{enumerate}
\end{cor}

\section*{Acknowledgements}

We wish to thank the anonymous referees for their careful reading of the paper, especially for the suggestion of adding some numerical evidence and extending the motivation of our work with the position \cite{CFMZS}. We also express our gratitude to Pavlo Yatsyna for writing the code in PARI/GP that counts the cardinalities of the sets $\cal{D}(r,u,m,n)$ and $\cal{D}_{r,u,m}(n)$ and lists all their elements. Moreover, we address our thanks to Bartosz Sobolewski for his remarks that improved this edition of the article.

\newpage

\section*{Appendix A}

We present the code written by Pavlo Yatsyna that we used to compute the quantities used in our paper. The functions \verb|D(r,u,m,n)| and \verb|DD(r,u,m,n)| write down the sets $\mathcal{D}(r,u,m,n)$ and $\mathcal{D}_{r,u,m}(n)$, respectively. Their cardinalities can be computed in two ways:
\begin{itemize}
    \item \verb|d(r,u,m,n)| and \verb|dd(r,u,m,n)| simply compute the number of elements of the sets;
    \item \verb|d_closed(r,u,m,n)| and \verb|dd_closed(r,u,m,n)| use the explicit formulae from Theorem \ref{main}\eqref{EqMain2} and Corollary \ref{exactd_rum(n)}.
\end{itemize}
The function \verb|C(r,u,m,n)| writes down comparisons between the cardinalities computed in these two ways.
\vspace{0.2cm}
\begin{lstlisting}
allocatemem(512000000);  \\ Allocate memory (~512 MB); adjust as needed for large permutations

\\ count pairs (x, y) with x <= r and ss[x] <= r, exactly u such pairs
number_of_pairs_xy(ss, r, u) =
{
  my(count = sum(k = 1, r, ss[k] <= r));
  return(count == u);
}

\\ count fixed points in positions r+1..r+n, exactly m of them
number_of_ts(ss, r, n, m) =
{
  my(count = sum(k = r+1, r+n, ss[k] == k));
  return(count == m);
}

\\ Check that no cycle (in the permutation ss) contains more than one element <= r
disjoint_cycles(ss, r) =
{
  local(Cycles = permcycles(ss));   \\ Compute the cycle decomposition of permutation ss
  return(
    vecmax(
      apply(
        y -> vecsum(apply(z -> z <= r, Vec(y))) > 1,
        Cycles
      )
    ) < 1
  );
}

\\ Compute the set D(r, u, m, n): permutations of size r + n with exactly u pairs (x,y) and m fixed points (ts) in positions r+1..r+n
D(r, u, m, n) =
{
  local(D = List());
  local(N = factorial(r + n));
  for(i = 1, N,
    local(ss = numtoperm(r + n, i));
    if (number_of_pairs_xy(ss, r, u) && number_of_ts(ss, r, n, m),
      listput(D, Vec(ss))
    );
  );
  return(Vec(D));
}

\\ Cardinality of D(r,u,m,n):
d(r, u, m, n) = #D(r, u, m, n);


\\ Compute the set D_{r, u, m}(n): permutations satisfying cycle-disjointness plus u pairs among the <= r region, plus m fixed points in the tail
DD(r, u, m, n) =
{
  local(D = List());
  local(N = factorial(r + n));
  for(i = 1, N,
    local(ss = numtoperm(r + n, i));
    if (
      disjoint_cycles(ss, r) &&
      number_of_ts(ss, 0, r, u) &&        \\ pairs among 1..r (r used here)
      number_of_ts(ss, r, n, m)
    ,
      listput(D, Vec(ss))
    );
  );
  return(Vec(D));
}

\\ Cardinality of D_{r,u,m}(n):
dd(r, u, m, n) = #DD(r, u, m, n);

\\ Closed formula for d_{r,u,m}(n):
dd_closed(r, u, m, n) ={
  binomial(r, u) *
  binomial(n, m) *
  sum(
    j = r - u, n - m,
    binomial(j, r - u) *
   (n - m)! /(n - m - j)! *
    (-1)^(n - m - j)
  );}

\\ Closed formula for d(r,u,m,n):
d_closed(r,u,m,n)=r!*dd_closed(r,u,m,n)

\\ Comparison d(r, u, m, n) vs d_closed(r, u, m, n) and dd(r, u, m, n) vs dd_closed(r, u, m, n)
C(r,u,m,n)=print("d(",r,", ",u,", ",m,", ",n,") = ",d(r,u,m,n)," vs ",d_closed(r,u,m,n),", d_",r,",",u,",",m"(",n,") = ",dd(r,u,m,n)," vs ",dd_closed(r,u,m,n))
\end{lstlisting}

\end{document}